\documentclass[11 pt]{amsart}
\usepackage{amssymb}
\usepackage{amsmath}
\usepackage{amsfonts}
\usepackage{graphicx}
\usepackage{amsthm}
\usepackage{enumerate}
\usepackage[mathscr]{eucal}
\usepackage{verbatim}

 \theoremstyle{plain}
\newtheorem{theorem}{Theorem}[section]
\newtheorem{lemma}[theorem]{Lemma}

\numberwithin{equation}{section}
\theoremstyle{plain}

\theoremstyle{remark}

\def\bbR{{\mathbb {R}}}

\def\fG{\frak G}
\def\cG{\mathcal G}

\begin{document}

\date{July, 2011}

\title[]
{Exceptional sets of projections, unions of $k$-planes, and associated transforms}

\begin{abstract}
We prove a generalization of a result of Peres and Schlag on the dimensions of certain exceptional sets 
of projections and then apply it to a geometric problem.
\end{abstract}

\author[]
{Daniel M. Oberlin}

\address
{D. M.  Oberlin \\
Department of Mathematics \\ Florida State University \\
 Tallahassee, FL 32306}
\email{oberlin@math.fsu.edu}


\maketitle

\section{Introduction}

Let $\cG (n,\ell)$ denote the $\ell(n-\ell)$-dimensional
Grassman manifold of $\ell$-planes passing through the origin in $\bbR^n$.
The following is a restatement of an observation of Peres and Schlag (Proposition 6.1 in \cite{PS})
which generalizes earlier results of Kaufman \cite{K} and Falconer \cite{F}:

\begin{theorem}\label{PStheorem}
Suppose $n\geq 2$ and $1\leq \ell <n$ are integers. Suppose $0<\alpha <n$, $0<\sigma <\ell$, and
$\ell(n-\ell) +\sigma -\alpha <\beta<\ell(n-\ell)$.
Suppose that the nonnegative Borel measure $\lambda$ on $\cG (n,\ell )$ is $\beta$-dimensional 
in the sense that 
$\lambda \big(B(\pi ,r)\big)\lesssim r^\beta$ for $\pi\in\cG (n,\ell )$ and $r>0$. Suppose that $E\subset\bbR^n$ 
is a Borel set with Hausdorff dimension at least $\alpha$. 
For $\pi\in\cG (n,\ell )$
let $P_\pi$ be the orthogonal projection
of $\bbR^n$ onto $\pi$. Then for 
$\lambda$-almost all $\pi$, $P_\pi (E)$ has Hausdorff dimension at least $\sigma$. 
\end{theorem}
\noindent The first result of this note is that if $\lambda$ satisfies a stronger version of $\beta$-dimensionality
(see \eqref{lambdahypagain} below), then the hypothesis $\ell(n-\ell) +\sigma -\alpha <\beta$ 
in Theorem \ref{PStheorem} can be weakened to $(n-\ell) +\sigma -\alpha <\beta$. Next let $\fG (d,k)$
denote the $(k+1)(d-k)$-dimensional Grassman manifold of \textbf{all} $k$-planes in $\bbR^d$. Our 
second theorem, which we will deduce from the first by taking $n=(k+1)(d-k)$ and $\ell =d-k$, states 
that if $S\subset \fG(d,k)$ has dimension exceeding $(k+1)(d-k)-k$, then 
$
\cup_{\pi\in S} \,\pi
$
has positive Lebesgue measure in $\bbR^d$ - for example, the union of a $(2d-3+\epsilon)$-dimensional collection of lines in $\bbR^d$ has positive measure. (The much more difficult Kakeya conjecture is that the union of any 
collection of lines containing at least one line in each direction has full dimension.) Our last observation is an estimate for a $k$-plane transform which is natural in the context of the above-mentioned union problem. To state these results precisely, we introduce some notation.

We parametrize by $\bbR^{\ell(n-\ell )}$ a collection of projections equivalent to
almost all of the projections $\{P_\pi :\pi\in\cG (n,\ell)\}$:
write $x=(x_i^j )$, where $i=1,\dots ,\ell$, $j=1,\dots ,n-\ell$, for an element of $\bbR^{\ell (n-\ell )}$.
Let $P_x :\bbR^n \rightarrow \bbR^\ell$ be defined by
\begin{equation}\label{projdef}
P_x :p=(p_1 ,\dots ,p_n )\mapsto \big(p_1 +\sum_{j=1}^{n-\ell}x_1^j p_{\ell+j}\, ,\,
p_2 +\sum_{j=1}^{n-\ell}x_2^j p_{\ell+j}\, ,\dots ,\, p_\ell +\sum_{j=1}^{n-\ell}x_\ell^j p_{\ell+j}\big).
\end{equation}
For $\xi\in\bbR^\ell$, define $T_x :\bbR^\ell \rightarrow \bbR^{n-\ell}$ by 
\begin{equation*}
T_x \xi=\big(\sum_{i=1}^\ell x_i^1 \xi_i ,\,\dots ,\sum_{i=1}^\ell x_i^{n-\ell} \xi_i \big).
\end{equation*}
For later use we note the identity
\begin{equation}\label{relation}
\langle \xi , P_x p \rangle =\langle (\xi ,T_x \xi ),p \rangle .
\end{equation}
Following \cite{PS}, we will say that a set $F\subset\bbR^\ell$ has Sobolev dimension at least $\sigma$ if
$F$ carries a Borel probability measure $\nu$ such that 
\begin{equation*}
\int_{\bbR^\ell}|\widehat {\nu}(\xi )|^2 \,
\frac{d\xi}{(1+|\xi |)^{\ell -\sigma}}
 <\infty .
\end{equation*}

Our analog of Theorem \ref{PStheorem} is the following:

\begin{theorem}\label{projections}
Suppose $\lambda$ is a compactly-supported nonnegative Borel measure on $\bbR^{\ell(n-\ell)}$ which satisfies the condition
\begin{equation}\label{lambdahypagain}
\lambda\big(\{x\in\bbR^{\ell (n-\ell )}:|T_x \xi -p_2 |\leq r\}\big)\leq c\,  r^\beta
\end{equation}
for some $c>0$ and all $\xi\in\bbR^\ell$ with $|\xi |=1$, $p_2 \in\bbR^{n-\ell}$, and $r>0$.
Suppose $E\subset\bbR^n$ is a Borel set with Hausdorff dimension at least $\alpha$. 
Suppose 
\begin{equation}\label{sigmahyp}
n-\ell+\sigma -\alpha<\beta .
\end{equation}
Then for $\lambda$-almost all $x\in\bbR^{\ell(n-\ell)}$, $P_x (E)$ has Sobolev dimension at least $\sigma$.
\end{theorem}

Next we parametrize almost all of $\fG (d,k)$ as follows: 
write $y=(y^j_i )$ where $i=0,\dots ,k$, $j=1,\dots ,d-k$ for an element of $\bbR^{(k+1)(d-k)}$ 
and $\pi_y$ for the $k$-plane in $\bbR^d$ given by 
\begin{equation*}
\big\{\big(x_1 ,\dots ,x_k, \, y_0^1 +\sum_{i=1}^k x_i y^1_i ,\,\dots ,\, y_0^{d-k} +\sum_{i=1}^k x_i y^{d-k}_i \big):
x=(x_1 ,\dots ,x_k )\in \bbR^k \big\}.
\end{equation*}
%

\begin{theorem}\label{unions}
Suppose $S\subset \bbR^{(k+1)(d-k)}$ is a compact set with Hausdorff dimension $\alpha >(k+1)(d-k)-k$.
Then 
\begin{equation*}
\bigcup_{y\in S}\pi_y
\end{equation*}
has positive Lebesgue measure in $\bbR^d$.
\end{theorem}
\noindent An easy example, which we will describe in \S 4, shows that the conclusion of Theorem \ref{unions}
may fail if 
$\alpha \leq(k+1)(d-k)-k$. 

Here is our $k$-plane estimate:

\begin{theorem}\label{transforms}
With notation as above, suppose $f\in C_c (\bbR^d )$ and define the 
$k$-plane transform $T$ by
\begin{equation*}
Tf(y)=\int_{[0,1]^k}f\big(x_1 ,\dots ,x_k ,
y_0^1 +\sum_{i=1}^k x_i y^1_i ,\,\dots ,\, y_0^{d-k} +\sum_{i=1}^k x_i y^{d-k}_i \big)\,dx 
\end{equation*}
for $y\in\bbR^{(k+1)(d-k)}$. Suppose $\mu$
is a compactly-supported nonnegative Borel measure on $\bbR^{(k+1)(d-k)}$ satisfying 
\begin{equation*}
\mu \big(B(y,r)\big)\leq c_2 \, r^{\alpha}
\end{equation*}
for some $\alpha\in \big((k+1)(d-k)-k,(k+1)(d-k)\big)$, some $c_2 >0$, all $y\in\bbR^{(k+1)(d-k)}$, and all $r>0$.
Fix $\epsilon$ with $0<\epsilon <\alpha -(k+1)(d-k)+k$, and 
define $q$ by 
\begin{equation*}
\frac{1}{2}-\frac{1}{q}=\frac{\alpha -(k+1)(d-k)+k-\epsilon}{2(d-k)}.
\end{equation*}
Then there is the estimate
\begin{equation}\label{Test}
\|Tf\|_{L^2 (\mu )}\leq C\, \|f\|_{L^2_x (L^{q'}_{x'})},
\end{equation}
where 
we write an element of $\bbR^d$ as $(x,x' )\in \bbR^k \times \bbR^{d-k}$ and where
$q'$ is the exponent conjugate to $q$.
Here $C$ depends only on $c_2$, $\epsilon$, and the diameter of the support of $\mu$.
\end{theorem}

The remainder of this note is organized as follows: \S2 contains the statement and proof of a lemma, 
\S3 contains the proofs of Theorems \ref{projections}, \ref{unions}, and \ref{transforms}, and \S4 contains 
some comments.

\section{ lemma}

\begin{lemma}\label{mainlemma} Suppose $\lambda$ is a compactly-supported nonnegative Borel measure on $\bbR^{\ell(n-\ell)}$ which satisfies the condition
\begin{equation}\label{lambdahyp}
\lambda\big(\{x\in\bbR^{\ell (n-\ell )}:|T_x \xi -p_2 |\leq r\}\big)\leq c_1 \, r^\beta
\end{equation}
for some $c_1 >0$, all $\xi\in\bbR^\ell$ with $|\xi |=1$, all $p_2\in\bbR^{n-\ell}$, and all $r>0$. Suppose $\mu$
is a compactly-supported nonnegative Borel measure on $\bbR^{n}$ satisfying 
\begin{equation}\label{muhyp1}
\mu \big(B(p,r)\big)\leq c_2 \, r^{\alpha}
\end{equation}
for some $\alpha\in (0,n)$, some $c_2 >0$, all $p\in\bbR^n$, and all $r>0$.
Then, for $g\in L^2 (\mu )$, we have
\begin{equation*}
\int \int_{R\leq |\xi |\leq 2R}\big|\widehat{gd\mu } (\xi ,T_x \xi ) \big|^2 \, d\xi \, d\lambda (x) 
\lesssim R^{n-\alpha -\beta} \|g\|^2_{L^2 (\mu )},
\end{equation*}
where the implied constant depends only on $c_1$, $c_2$, and the diameters of the supports of 
$\lambda$ and $\mu$.
\end{lemma}

\begin{proof}
Fix $\kappa\in C^{\infty}_c (\bbR^n )$ with $\kappa =1$ on the support of $\mu$. 
Define
\begin{equation*}
\Gamma_R \doteq \big\{(\xi,T_x \xi):\xi\in\bbR^\ell, R\leq |\xi|\leq 2R,\ x\in\text{supp}(\lambda )\big\}.
\end{equation*}
Since 
\begin{equation*}
\int_{\bbR^n} |\widehat{\kappa}\big((\xi ,T_x \xi)-p\big)|\, dp \lesssim 1
\end{equation*}
we have
\begin{multline*}
\int_{R\leq |\xi |\leq 2R}\,\int \big|\widehat{gd\mu}(\xi ,T_x \xi)\big|^2 \, d\lambda (x)\, d\xi = \\
\int_{R\leq |\xi |\leq 2R}\,\int \Big|\int_{\bbR^n} \widehat{\kappa}\big((\xi ,T_x \xi)-p\big)\,\widehat{gd\mu}(p)\, dp\Big|^2 \, 
d\lambda (x)\, d\xi \lesssim \\
\int_{\bbR^n} \Big(\int_{R\leq |\xi |\leq 2R}\,\int\big|\widehat{\kappa}\big((\xi ,T_x \xi)-p\big)\big|\,
d\lambda (x)\,d\xi\Big) \big|\widehat{gd\mu}(p)|^2 \,dp.
\end{multline*}
Let $M=M_1 +M_2 +M_3$ where $M_1 >n$, $M_2 >\beta$, and $M_3 >\ell$.
We write $p=(p_1 ,p_2 )$ with $p_1 \in \bbR^\ell$, $p_2 \in \bbR^{n-\ell}$ and estimate
\begin{multline*}
\int_{R\leq |\xi |\leq 2R}\,\int\big|\widehat{\kappa}\big((\xi ,T_x \xi )-p\big)\big|\,d\lambda (x)\,d\xi \lesssim
\int_{R\leq |\xi |\leq 2R}\,\int\frac{1}{(1+|(\xi ,T_x \xi )-p |)^{M}}\,d\lambda (x)\,d\xi \lesssim \\
\frac{1}{\big(1+\text{dist}(\Gamma_R ,p)\big)^{M_1}}
\int_{R\leq |\xi |\leq 2R}\Big(\int_{\bbR^{\ell(n-\ell)}}
\frac{1}{(1+|T_x  \xi -p_2 |)^{M_2}}
\,d\lambda (x)\Big)
\frac{1}{(1+|\xi-p_1 |)^{M_3}}\, d\xi\lesssim 
\end{multline*}
\begin{equation*}
\frac{1}{R^\beta\,\big(1+\text{dist}(\Gamma_R ,p)\big)^{M_1}},
\end{equation*}
where we have used $|\xi|\geq R$ and \eqref{lambdahyp} in bounding the $x$-integral. Thus
\begin{equation*}
\int_{R\leq |\xi |\leq 2R}\,\int \big|\widehat{gd\mu}(\xi ,T_x \xi)\big|^2 \, d\lambda (x)\, d\xi \lesssim
\frac{1}{R^\beta}\int_{\bbR^n}\frac{|\widehat{gd\mu}(p)|^2}{\big(1+\text{dist}(\Gamma_R ,p)\big)^{M_1}}\, dp
\end{equation*}
and so
\begin{multline}\label{est1}
\int_{R\leq |\xi |\leq 2R}\,\int \big|\widehat{gd\mu}(\xi ,T_x \xi)\big|^2 \, d\lambda (x)\, d\xi \lesssim \\
\frac{1}{R^\beta}\int_{B(0,c_3 R)}|\widehat{gd\mu}(p)|^2 \, dp \,+
\frac{\|g\|_{L^2 (\mu )}^2}{R^\beta} \int_{\{|p|\geq c_3 R\}}
\frac{1}{\big(1+\text{dist}(\Gamma_R ,p)\big)^{M_1}}\, dp,
\end{multline}
where $c_3$, depending on the diameter of the support of $\lambda$, is such that $|p|\geq c_3 R$ implies 
$\text{dist} (\Gamma_R ,p)\geq |p|$. 
We estimate the first term by duality:
\begin{multline}\label{est1.7}
\Big(\int_{B(0,c_3 R)}|\widehat{gd\mu}(p)|^2 \, dp\Big)^{1/2}=
\sup \Big\{\Big|\int_{B(0,c_3 R)}f(p)\, \widehat{gd\mu}(p)\,dp\Big|:\|f\|_{L^2 (B(0,c_3 R))}\leq 1
\Big\}\leq\\
\sup \Big\{\Big|\int_{\bbR^n}\widehat{f}(p)\, g(p)\,d\mu (p)\Big|:\|f\|_{L^2 }\leq 1,
\,\text{supp}(f)\subset B(0,c_3 R)
\Big\}
\lesssim \\
\|g\|_{L^2 (\mu )}
\sup \Big\{\Big(\int_{\bbR^n}|\widehat{f}(p)|^2\, d\mu (p)\Big)^{1/2}:\|f\|_{L^2 }\leq 1,
\,\text{supp}(f)\subset B(0,c_3 R)
\Big\}.
\end{multline}
Fix a Schwartz function $\phi$ on $\bbR^n$ such that $\phi (p)=1$ for $|p|\leq c_3 R$ and
\begin{equation}\label{hyp2}
\text{supp}(\widehat{\phi})\subset B(0,1/10R),\ \|\widehat{\phi}\|_{\infty}\lesssim R^{n}.
\end{equation}
Suppose $f$ is supported in $B(0,c_3 R)$ and satisfies $\|f\|_2 =1$. Since $\widehat{f}=\widehat{f}\ast\widehat{\phi}$ we have
\begin{equation*}
|\widehat{f}|\leq (|\widehat{f}|^2 \ast |\widehat{\phi}|)^{1/2}\,\|\widehat{\phi}\|_1^{1/2}\lesssim
(|\widehat{f}|^2 \ast |\widehat{\phi}|)^{1/2}\
\end{equation*}
and so 
\begin{multline*}
\int_{\bbR^n}|\widehat{f}(p)|^2\, d\mu (p)\lesssim
\int_{\bbR^n}(|\widehat{f}|^2 \ast |\widehat{\phi}|)(p)\, d\mu (p)= \\
\int_{\bbR^n}|\widehat{f}(q)|^2 \int_{\bbR^n} |\widehat{\phi}(p-q)|\,d\mu (p)\,dq \lesssim
R^{n-\alpha}\int_{\bbR^n}|\widehat{f}(q)|^2\, dq =R^{n-\alpha}
\end{multline*}
since, by \eqref{muhyp1} and \eqref{hyp2}, the $p$-integral in the third from last term is 
$\lesssim R^{n-\alpha}$. Now \eqref{est1.7} implies that
\begin{equation}\label{est1.4}
\int_{B(0,c_3 R)}|\widehat{gd\mu}(p)|^2 \, dp \lesssim \|g\|_{L^2 (\mu )}^2 \, R^{n-\alpha}.
\end{equation}
Since $|p|\geq c_3 R$ implies 
 $\text{dist}(\Gamma_R ,p)\gtrsim |p|$,
\begin{equation}\label{est1.5}
 \int_{\{|p|\geq c_3 R\}}
\frac{1}{\big(1+\text{dist}(\Gamma_R ,p)\big)^{M_1}}\, dp \lesssim
\int_{\bbR^n}\frac{1}{\big(1+|p|\big)^{M_1}}\, dy <
\infty ,
\end{equation}
where the last inequality follows because $M_1 >n$. From \eqref{est1}, \eqref{est1.4}, 
\eqref{est1.5}, and $\alpha <n$ we then have 
\begin{equation*}
\int_{R\leq |\xi |\leq 2R}\,\int \big|\widehat{gd\mu}(\xi ,T_x \xi)\big|^2 \, d\lambda (x)\, d\xi \lesssim 
R^{n-\alpha -\beta}\,\|g\|_{L^2 (\mu )}^2 
\end{equation*}
as desired.
\end{proof}

\section{Proofs of the theorems}

\noindent{\it Proof of Theorem \ref{projections}:}
Because of the strict inequality in \eqref{sigmahyp} we can assume that $E$ carries a 
compactly-supported Borel probability measure $\mu$ satisfying 
\begin{equation*}
\mu \big(B(p,r)\big)\leq c_2 \, r^{\alpha}
\end{equation*}
for some $c_2 >0$ and all $p\in\bbR^n$ and all $r>0$.
Write $P_x (\mu )$ for the push-forward measure on $\bbR^\ell$ given by 
\begin{equation*}
\int_{\bbR^\ell} f\,dP_x (\mu )=\int_{\bbR^{n}}f\big( P_x p\big)\,d\mu (p). 
\end{equation*}
It is enough to show that 
\begin{equation}\label{equ1}
\int_{\bbR^{\ell(n-\ell)}}\int_{\bbR^\ell}|\widehat{P_x (\mu )}(\xi )|^2 \, \frac{d\xi}{(1+|\xi |)^{\ell-\sigma }}
\,d\lambda (x)<\infty .
\end{equation}
But, by  \eqref{relation},
\begin{equation*}
\widehat{P_x (\mu )}(\xi )=\int_{\bbR^n}
e^{-2\pi i \langle \xi ,P_x p\rangle}
\,d\mu (p) =\int_{\bbR^n}
e^{-2\pi i \langle (\xi ,T_x \xi ) ,p\rangle}
\,d\mu (p)
=
\widehat{\mu}(\xi ,T_x \xi ).
\end{equation*}
Thus 
\begin{equation*}
\int_{\bbR^{\ell(n-\ell)}}\int_{2^j \leq |\xi |\leq 2^{j+1}}
|\widehat{P_x (\mu )}(\xi )|^2 \, \frac{d\xi}{(1+|\xi |)^{\ell-\sigma }}\, d\lambda (x)
\lesssim 2^{(n-\ell+\sigma -\alpha -\beta)j}
\end{equation*}
by Lemma \ref{mainlemma}, and then \eqref{equ1} follows from \eqref{sigmahyp}.

\

\noindent{\it Proof of Theorem \ref{unions}:}
It is enough to observe that for almost every $$x=(x_1 ,\dots, x_k )\in[0,1]^k$$ the set 
\begin{equation}\label{xsection}
\big\{ (y_0^1 +\sum_{i=1}^k x_i y^1_i \,,\dots ,\, y_0^{d-k} +\sum_{i=1}^k x_i y^{d-k}_i ):
y\in S\big\}
\end{equation}
has Sobolev dimension exceeding $d-k$. We will deduce this from Theorem \ref{projections}.

In Theorem \ref{projections} we set $n=(k+1)(d-k)$, set $\ell =d-k$, and take $\lambda$ to be $k$-dimensional Lebesgue measure 
on a copy of $[0,1]^k$ embedded in $\bbR^{\ell(n-\ell)}$ in such a way that the mapping $P_x$ in \eqref{projdef} 
is given by 
\begin{equation}\label{Pdef2}
P_x y =(y_0^1 +\sum_{i=1}^k x_i y^1_i ,\,\dots ,\, y_0^{d-k} +\sum_{i=1}^k x_i y^{d-k}_i )
\end{equation}
for $x=(x_1 ,\dots ,x_k )\in [0,1]^k$. We will note below that \eqref{lambdahypagain} is satisfied with $\beta =k$. It will then follow from our hypothesis $\alpha >(k+1)(d-k)-k$ that $\ell-n+\alpha +\beta>d-k$. Therefore  
$\sigma$ in \eqref{sigmahyp} may be chosen to satisfy $\sigma >d-k$. It will then follow from
Theorem \ref{projections} that the Sobolev dimension of \eqref{xsection} 
exceeds $d-k$ for almost all $x\in [0,1]^k$.

To verify \eqref{lambdahypagain}, we begin by noting that \eqref{Pdef2} and \eqref{relation} imply 
\begin{equation*}
T_x \xi =(x_1 \xi ,\dots ,   x_k \xi ).
\end{equation*}
Since $|\xi |=1$ implies that $\{x_j \in [0,1]:|x_j \xi -\eta |<r\}$ has one-dimensional Lebesgue measure 
$\lesssim r$ for any $\eta\in\bbR^{d-k}$, it follows from the definition of $\lambda$ that \eqref{lambdahypagain} is satisfied with $\beta =k$.

\

\noindent{\it Proof of Theorem \ref{transforms}:}
To establish \eqref{Test} by duality it is enough to show that 
\begin{equation*}
\int_{[0,1]^k}\|T^* f(x,x' )\|^2_{L^q_{x'}}\,dx\leq C\, \|g\|^2_{L^2 (\mu )}.
\end{equation*}
We will apply Lemma \ref{mainlemma} as in the proof of Theorem \ref{projections}.

As above, set $n=(k+1)(d-k)$, set $\ell =d-k$, and take $\lambda$ to be $k$-dimensional Lebesgue measure 
on a copy of $[0,1]^k$ embedded in $\bbR^{\ell(n-\ell)}$ in such a way that the mapping $P_x$ 
is given by \eqref{Pdef2}
%
%
for $x\in [0,1]^k$. We recall that \eqref{lambdahypagain} is satisfied with $\beta =k$.

Now 
\begin{multline*}
\langle f,T^* g\rangle = \langle Tf,g\rangle=\\
\int_{\bbR^{(k+1)(d-k)}}\int_{[0,1]^k}f\big(x_1 ,\dots ,x_k ,
y_0^1 +\sum_{i=1}^k x_i y^1_i ,\,\dots ,\, y_0^{d-k} +\sum_{i=1}^k x_i y^{d-k}_i \big)\,dx\, g(y)\,d\mu (y).
\end{multline*}
Thus for $x\in[0,1]^k$ and $\xi\in\bbR^{d-k}$ we have
by \eqref{Pdef2} and \eqref{relation}
\begin{multline*}
\widehat{T^* g (x,\cdot )}(\xi )=
\int_{\bbR^{(k+1)(d-k)}}
e^{-2\pi i \langle \xi , P_x (y)\rangle}\,g(y)\,d\mu (y)=\\
\int_{\bbR^{(k+1)(d-k)}}
e^{-2\pi i \langle (\xi ,T_x \xi ), y\rangle }\,g(y)\,d\mu (y)=\widehat{gd\mu}(\xi ,T_x \xi )
\end{multline*}
and so 
\begin{multline*}
\int_{[0,1]^k}\int_{\bbR^{d-k}}|\widehat{T^* g (x,\cdot )}(\xi )|^2 \, |\xi|^{\alpha-(k+1)(d-k)+k-\epsilon} \,d\xi\, dx =\\
\int_{[0,1]^k}\int_{\bbR^{d-k}}|\widehat{gd\mu}(\xi ,T_x \xi )|^2 \,  |\xi|^{\alpha-(k+1)(d-k)+k-\epsilon} \,d\xi\, dx .
\end{multline*}
Since
\begin{equation*}
\int_{[0,1]^k}\int_{2^j \leq |\xi |\leq 2^{j+1}}
|\widehat{gd\mu}(\xi ,T_x \xi )|^2 \, d\xi\, dx \lesssim 2^{j((k+1)(d-k)-k-\alpha )}\,\|g\|^2_{L^2 (\mu )}
\end{equation*}
by Lemma \ref{mainlemma}, it follows that 
\begin{equation*}
\int_{[0,1]^k}\int_{\bbR^{d-k}}
|\widehat{T^* g (x,\cdot )}(\xi )|^2
 \,  |\xi|^{\alpha-(k+1)(d-k)+k-\epsilon} \,d\xi\, dx
\lesssim  \|g\|^2_{L^2 (\mu )}.
\end{equation*}
Then fractional integration shows that 
\begin{equation*}
\int_{[0,1]^k}\| T^* g (x,x' )\|^2 _{L^q_{x'}}\, dx \lesssim \|g\|^2_{L^2 (\mu )}
\end{equation*}
as desired.

\section{Comments}

\noindent (i) Here is an example which shows that the hypothesis $\alpha >(k+1)(d-k)-k$ in Theorem \ref{unions} is necessary. 
Fix a compact set $K\subset [0,1]$ having Hausdorff dimension $1$ and (one-dimensional) 
Lebesgue measure $0$. With the notation of Theorem \ref{unions}, the set $S$
given by 
\begin{multline*}
S=\big\{y=(y_i^j ):y_i^j \in [0,1] \text{ if } 0\leq i\leq k, 1\leq j \leq d-k-1;\\
 y_0^{d-k}\in K;\, y_i^{d-k}=0 \text{ if }1\leq i\leq k
\big\}
\end{multline*}
has Hausdorff dimension $(k+1)(d-k-1)+1 =(k+1)(d-k)-k$, while $\cup_{y\in S}\,\pi_y$ has $d$-dimensional 
Lebesgue measure $0$.

\ 

\noindent (ii) Theorem \ref{unions} can easily be deduced directly from Theorem \ref{transforms} (by
applying the latter theorem with $f$ the indicator function of $\cup_{y\in S}\,\pi_y$). Our deduction
of Theorem \ref{unions} from Theorem \ref{projections} represents a different way to view such problems.

\

\noindent (iii) With the notation from Theorem \ref{unions}, it is not difficult to show (the not usually sharp result) that if $S$ has Hausdorff dimension at least $\alpha$, then $\cup_{y\in S}\,\pi_y$ has Hausdorff dimension at least 
$\min\{2k+\alpha -k(d-k),d\}$:
the proof of Theorem \ref{unions} shows that for almost every $x\in  [0,1]^k$, the 
$x$-section \eqref{xsection} of $\cup_{y\in S}\,\pi_y$ 
has Sobolev dimension at least $k+\alpha -k(d-k)$. The conclusion then follows from a higher dimensional 
analogue of Theorem 5.8 in \cite{F2}. 

\

\noindent (iv) Suppose $A_0 ,A_1  \subset \bbR^n$ are compact sets with Hausdorff dimensions $\alpha_1 ,
\alpha_2$. Then, for almost every $x\in\bbR$, $A_0 +xA_1$ has Sobolev dimension at 
least $\alpha_0 +\alpha_1 +1 -n$. To see this, take $k=1$ and $d=n+1$ in the paragraph above. If $S$ is the set 
$\{y=(y_0 ,y_1 ):y_i \in A_i\}$, then the Hausdorff dimension $\alpha$ of $S$ is at least $\alpha_0 +\alpha_1$
and the $x$-sections \eqref{xsection} are the sets $A_0 +xA_1$ .

\

\noindent (v) Without the hypothesis $\alpha >(k+1)(d-k)-k$, the proof of Theorem \ref{transforms} yields 
the estimate
\begin{equation*}
\|Tf\|_{L^2 (\mu )}\leq C\, \|f\|_{L^2_x (W^{2,-\rho /2}_{x'})}
\end{equation*}
whenever $\rho <(k+\alpha -(k+1)(d-k))/2$.

\

\noindent (vi) Theorems \ref{unions} and \ref{transforms} generalize Theorems $3_H$ and $4_H$ from \cite{O}.

\

\noindent (vii) The author learned the (fairly simple) techniques appearing in the 
proof of Lemma \ref{mainlemma} from the papers \cite{W} and \cite{E}.

\end{document}